\documentclass[reqno,11pt]{article}
\usepackage{a4wide,color,eucal,enumerate,mathrsfs}
\usepackage[normalem]{ulem}
\usepackage{psfrag}          
\usepackage{amsmath,amssymb,epsfig}
\numberwithin{equation}{section}

\usepackage[latin1]{inputenc}

\usepackage[pdfborder={0 0 0}]{hyperref}

\newenvironment{proof}{\removelastskip\par\medskip   
\noindent{\em proof} \rm}{\penalty-20\null\hfill$\square$\par\medbreak}

\newtheorem{theorem}{Theorem}[section]

\newtheorem{lemma}[theorem]{Lemma}
\newtheorem{proposition}[theorem]{Proposition}

\newtheorem{example}[theorem]{Example}

\newcommand{\CD}{{\sf CD}}
\newcommand{\RCD}{{\sf RCD}}

\newcommand{\ppi}{{\mbox{\boldmath$\pi$}}}
\newcommand{\ggamma}{{\mbox{\boldmath$\gamma$}}}

\newcommand{\sggamma}{{\mbox{\scriptsize\boldmath$\gamma$}}}

\newcommand{\supp}{\mathop{\rm supp}\nolimits}
\renewcommand{\d}{{\mathrm d}}
\newcommand{\N}{\mathbb{N}}
\newcommand{\Z}{\mathbb{Z}}
\newcommand{\R}{\mathbb{R}}
\newcommand{\LL}{\mathscr{L}}

\newcommand{\mm}{\mathfrak m}

\newcommand{\sfd}{{\sf d}}
\newcommand{\prob}[1]{\mathscr P(#1)}
\newcommand{\probt}[1]{\mathscr P_2(#1)}
\newcommand{\Op}{\rm Opt}

\newcommand{\geo}{{\rm{Geo}}}
\newcommand{\e}{{\rm{e}}}
\newcommand{\id}{{\rm{id}}}
\newcommand{\gopt}{{\rm{OptGeo}}}

\newcommand{\midp}{{\rm Mid}}

\newcommand{\restr}[1]{\lower3pt\hbox{$|_{#1}$}}
\newcommand{\ent}{\mathrm{Ent}}

\title{Failure of the local-to-global property for $CD(K,N)$ spaces}
\begin{document}

\author{
   Tapio Rajala
   \thanks{University of Jyv\"askyl\"a, \textsf{tapio.m.rajala@jyu.fi}}
   }
\maketitle

\begin{abstract}
 Given any $K \in \R$ and $N \in [1,\infty]$ we show that there exists a compact geodesic metric measure space satisfying 
 locally the $\CD(0,4)$ condition but failing $\CD(K,N)$ globally.
 The space with this property is a suitable non convex subset of $\R^2$ equipped with the $l^\infty$-norm 
 and the Lebesgue measure. Combining many such spaces gives a (non compact) complete geodesic metric measure space satisfying 
 $\CD(0,4)$ locally but failing $\CD(K,N)$ globally for every $K$ and $N$. 
\end{abstract}

\tableofcontents

\section{Introduction}

In \cite{Lott-Villani09, Sturm06I, Sturm06II} Lott, Sturm and Villani proposed a definition of Ricci curvature lower bounds in metric measure spaces.
The definitions were in terms of convexity properties of functionals in the space of probability measures.
The most relevant definition in the context of this paper is the $\CD(0,N)$ condition, with $0$ taking the place of 
a lower bound on the curvature, which is usually denoted by $K \in \R$ in the more general definition
($K=0$ here means non negative Ricci curvature), and $N < \infty$ being the upper bound
on the dimension of the space.
The $\CD(0,N)$ condition on a metric measure space $(X,\sfd,\mm)$ requires that between any two probability measures on
the space there exists at least one geodesic along which the entropy
\[
 \ent_N(\rho\mm) = -\int_X\rho^{1-\frac1{N'}}\d\mm
\]
is convex for all $N' \ge N$. (See Section \ref{sec:preli} for more details.)

Soon after the definition of $\CD(0,N)$ had been introduced it was noticed that $\R^n$ equipped with any norm and with
the Lebesgue measure satisfies $\CD(0,n)$. See the end of Villani's book \cite{Villani09} for an outline of the proof of this fact.
In particular we have:

\begin{theorem}[Cordero-Erausquin, Sturm and Villani]\label{thm:CD02maps}
The space $(\R^2,||\cdot||_\infty,\LL_2)$ satisfies $\CD(0,2)$.
\end{theorem}

A problematic feature of spaces like $(\R^2,||\cdot||_\infty,\LL_2)$ is that between most of the points there exist
a huge number of geodesics joining them. In particular, there are a lot of branching geodesics. Initially many results
for $\CD(K,N)$ spaces were proven under the assumption that there are no branching geodesics. Later some of these results have been
proven without such assumption (for instance local Poincar\'e inequalities \cite{R2012a, R2012b}). In some results
the general case with branching geodesics remains open. Branching geodesics are also known to exist, for example,
in some positively curved $\CD(K,N)$ spaces, see the recent paper by Ohta \cite{Ohta2013}.

Until now one of the basic open questions for general $\CD(K,N)$ spaces was the local-to-global property
of the $\CD(K,N)$ condition. It is known that under the non branching assumption assuming $\CD(0,N)$ (or $\CD(K,\infty)$
or $\CD^*(K,N)$) to hold locally (i.e. in a neighbourhood of any point) is the same as assuming it to hold globally.
For $\CD(K,\infty)$ this was proven by Sturm \cite{Sturm06I}, for $\CD(0,N)$ by Villani \cite{Villani09}, and for
$CD^*(K,N)$ by Bacher and Sturm \cite{BacherSturm10}.
Such property is natural to expect from an abstract notion of Ricci curvature lower bounds - after all, the
classical definition is local.
The notion $\CD^*(K,N)$ refers to the reduced curvature-dimension condition. It is not (at least a priori) as restrictive
as the $\CD(K,N)$ condition, but it is more natural in the local-to-global questions.

\begin{figure}
   \centering
   \includegraphics[width=0.9\textwidth]{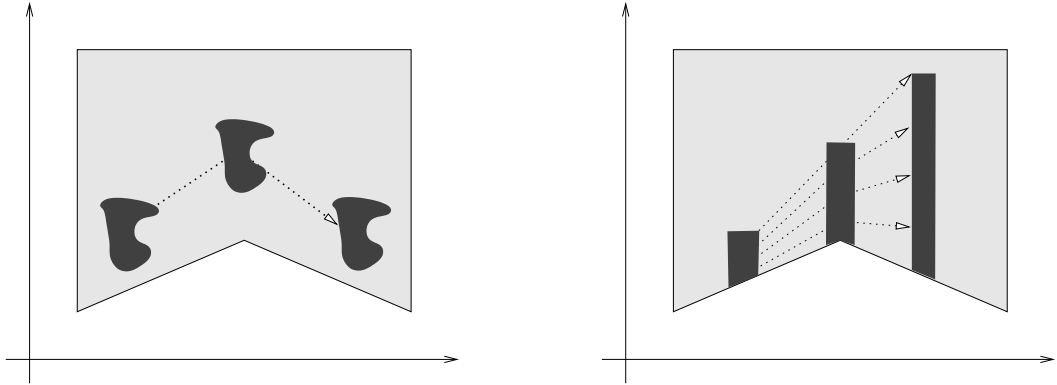}
   \caption{On the left we see how a measure can be easily transported around a corner when our distance is
            given by the $l^\infty$-norm. On the right we have the extremal case when we go around a corner. 
            In this case we have to squeeze the measure a bit.}
  \label{fig:idea}
 \end{figure}

In this paper we show that not even $\CD(0,N)$ does, in general, have the local-to-global property.
The idea of our example showing that local $\CD(0,N)$ does not imply global $\CD(0,N)$ is surprisingly simple.
One starts with the observation, which we already mentioned, that $(\R^2,||\cdot||_\infty)$ has lots of geodesics. 
There are even so many geodesics that one can go around some Euclidean corners with them.
Therefore we at least have domains in $\R^2$ that are not convex in the Euclidean sense but still (weakly) geodesically
convex with the $l^\infty$-norm.
Next we observe that we can locally move two identical objects around a corner, see the left picture in Figure
\ref{fig:idea}. This roughly means that moving measures that are approximately the same should not be a problem in view of
the local $\CD(0,2)$ condition.

For more general sets the $45$ degree angle gives the extremal case when going around a corner. See the right picture in
Figure \ref{fig:idea} for the extremal case. There we have to shrink the measure in the vertical direction when we move it
around the corner.
This suggests that we have to give up our hope on $\CD(0,2)$. Still the particular transport seems to satisfy $\CD(0,4)$,
for instance. However, when we take thinner and thinner strips closer and closer to the corner we notice that the
estimates do not scale property. An obvious idea to correct this
is to smoothen the corner, and in fact replacing the corner with a piece of a circle will do:

\begin{example}\label{ex:1}
 Let $K \in \R$ and $N \in [1,\infty]$. 
There exists a compact geodesic metric measure space $(X,\sfd,\mm)$ satisfying $\CD(0,4)$ locally, but failing to satisfy $\CD^*(K,N)$ 
(and $\CD(K,N)$) globally.
Take $X$ to be the closed subset of $\R^2$ shown in Figure \ref{fig:example}.
(We shall specify it more carefully in Section \ref{sec:verify}.)
As the distance take $\sfd(x,y) = ||x-y||_\infty$
and as the reference measure the restriction of the Lebesgue measure $\mm = \LL_2\restr{X}$.
 \begin{figure}
   \centering
   \includegraphics[width=0.6\textwidth]{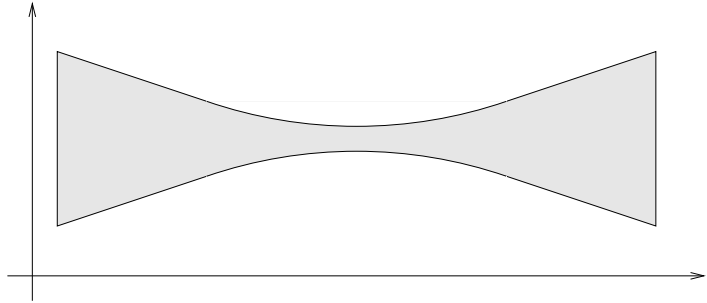}
   \caption{An illustration of the space $X$ of Example \ref{ex:1} as a compact subset of $\R^2$ with the $l^\infty$-norm.
            The space satisfies $\CD(0,4)$ locally, but not globally.}
  \label{fig:example}
 \end{figure}
\end{example}

We note that if in Example \ref{ex:1} we were to drop either the requirement that $(X,\sfd)$ is complete 
or the requirement that it is geodesic the example would be close to trivial. However, with both of these assumptions
in place, if we want to get the example as a subset of $\R^2$,
we are forced to consider optimal transport at and near the boundary of a non convex set. 
Verifying the $\CD(0,4)$ condition at the boundary turned out to require some calculations.

Indeed, proving that $(X,\sfd,\mm)$ locally satisfies $\CD(0,4)$ takes most of this paper whereas the failure of global $\CD^*(K,N)$
follows immediately by considering optimal transport between measures with large supports on the opposite sides of the 'neck'.
Gluing together infinitely many spaces of the type shown in Example \ref{ex:1} gives a (non compact) complete geodesic
metric measure space satisfying $\CD(0,4)$ locally but
failing the global $\CD(K,N)$ for any $K \in \R$ and $N \in [1,\infty]$.

Although $\CD(K,N)$ (and $\CD^*(K,N)$) fails to have the local-to-global property,
the more recent definition of Riemannian Ricci curvature bounds by Ambrosio, Gigli and Savar\'e
\cite{AmbrosioGigliSavare11-2} (see also \cite{AmbrosioGigliMondinoRajala12} for some generalization and simplifications
and \cite{Erbar-Kuwada-Sturm13, AMS} for the finite dimensional definitions),
 $\RCD^*(K,N)$ for short, could still have the local-to-global property. 
The fact that $\RCD^*(K,N)$ spaces are essentially non branching and there exist optimal maps from absolutely continuous measures
\cite{RajalaSturm12, Gigli12a, GigliRajalaSturm13} strongly supports this conjecture.

\smallskip
\noindent {\bf Acknowledgements.}
The author is grateful for the many enlightening discussions with Luigi Ambrosio and Nicola Gigli on this subject.
The author also acknowledges the financial support of the Academy of Finland project no. 137528.

\section{Preliminaries}\label{sec:preli}

In this paper the norm we mostly use is the $l^\infty$-norm and hence we sometimes abbreviate
 $||(x_0,y_0)-(x_1,y_1)|| := ||(x_0,y_0)-(x_1,y_1)||_\infty := \max\{|x_0-x_1|,|y_0-y_1|\}$.
We denote the Euclidean norm in $\R$ by $|\cdot|$.

\subsection{Optimal mass transportation}

We will give here only a few facts about optimal mass transportation. For a more detailed introduction
we refer to the books by Villani \cite{Villani09} and by Ambrosio and Gigli \cite{AmbrosioGigli11}.
We denote by $\prob X$ the space of Borel probability measures on the complete
and separable metric space $(X,\sfd)$ and by $\probt X \subset \prob X$ the subspace
consisting of all the probability measures with finite second moment. Our example $X$ is compact
and thus for it we have $\probt X = \prob X$. However, in general the measures with finite second moment are considered
in order to have finite $W_2$-distance between the measures (see below for the definition of the distance $W_2$).

Given two probability measures $\mu_0,\mu_1 \in \prob X$ and a Borel cost function $c \colon X \times X \to [0,\infty]$
the optimal mass transportation problem is to minimize
\begin{equation}\label{eq:generalmin}
 \int_X c(x,y)\,\d\ggamma(x,y)
\end{equation}
among all $\ggamma \in \prob{X \times X}$ with $\mu_0$ and $\mu_1$ as the first and the second marginal.

In the definition of the Ricci curvature lower bounds we will use the quadratic transportation distance $W_2(\mu_0,\mu_1)$,
which is given by the cost function $c(x,y) = \sfd(x,y)^2$. In other words, 
for $\mu_0,\mu_1 \in \probt X$ it is defined by
\begin{equation}\label{eq:Wdef}
  W_2^2(\mu_0,\mu_1) = \inf_\sggamma \int_X \sfd^2(x,y) \,\d\ggamma(x,y),
\end{equation}
where again the infimum is taken over all $\ggamma \in \prob{X \times X}$ with $\mu_0$ and $\mu_1$ as the first and the second marginal.
Assuming the space $(X,\sfd)$ to be geodesic, also the space $(\probt X, W_2)$ is geodesic.
We denote by $\geo(X)$ the space of (constant speed minimizing) geodesics on $(X,\sfd)$.
The notation $\e_t:\geo(X)\to X$, $t\in[0,1]$ is used for the evaluation 
maps defined by $\e_t(\gamma):=\gamma_t$.
A useful fact is that any geodesic $(\mu_t) \in \geo(\probt X)$
can be lifted to a measure $\ppi \in \prob{\geo(X)}$, so that $(\e_t)_\#\ppi = \mu_t$ for all $t \in [0,1]$.
Given $\mu_0,\mu_1\in\probt X$, we denote by $\gopt(\mu_0,\mu_1)$ the space of all
$\ppi \in \prob{\geo(X)}$ for which $(\e_0,\e_1)_\#\ppi$ realizes the minimum in \eqref{eq:Wdef}.


A property of optimal transport plans that we will frequently use is cyclical monotonicity.
It holds in a great generality, and in particular in the minimization problems we are considering
in this paper.
 A set $\Gamma \subset X \times X$ is called $c$-cyclically monotone if for any $k \in \N$
and $(x_1,y_1), \dots, (x_k,y_k) \in \Gamma$ we have
\[
 \sum_{i=1}^k c(x_i,y_i) \le \sum_{i=1}^k c(x_i,y_{i+1})
\]
with the identification $y_{k+1} = y_1$.
Now, given $\mu_0,\mu_1\in \probt X$
and an optimal transport plan $\ggamma$ minimizing \eqref{eq:generalmin} there exists a $c$-cyclically monotone subset
$\Gamma $ with full $\ggamma$-measure.

\subsection{Ricci curvature lower bounds in metric measure spaces}

We will define here the $\CD^*(K,N)$ condition, coming from the paper by Bacher and Sturm \cite{BacherSturm10},
and not the $\CD(K,N)$ condition.
The reason for this is that in the non branching case the $\CD^*(K,N)$ condition has the local-to-global property.
Moreover, for $K \ge 0$ and $N \in [1,\infty)$ we have
\[
 \CD(K,N) \Rightarrow \CD^*(K,N) \Rightarrow \CD(\frac{N-1}NK,N).
\]
For the proof of this and for a more detailed discussion of the relation with $\CD^*(K,N)$ and $\CD(K,N)$ we refer to \cite{BacherSturm10} 
(see also the papers by Cavalletti and Sturm \cite{Cavalletti-Sturm12} and by Cavalletti \cite{Cavalletti12}). 
Since we show that our example fails $\CD(K,\infty)$, it will also fail $\CD(K,N)$ and $\CD^*(K,N)$ for all $N$.

Given $K \in \R$ and $N \in [1, \infty)$, we define the distortion coefficient $[0,1]\times\R^+\ni (t,\theta)\mapsto \sigma^{(t)}_{K,N}(\theta)$ as
\[
\sigma^{(t)}_{K,N}(\theta):=\left\{
\begin{array}{ll}
+\infty,&\qquad\textrm{ if }K\theta^2\geq N\pi^2,\\
\frac{\sin(t\theta\sqrt{K/N})}{\sin(\theta\sqrt{K/N})}&\qquad\textrm{ if }0<K\theta^2 <N\pi^2,\\
t&\qquad\textrm{ if }K\theta^2=0,\\
\frac{\sinh(t\theta\sqrt{K/N})}{\sinh(\theta\sqrt{K/N})}&\qquad\textrm{ if }K\theta^2 <0.
\end{array}
\right.
\]

Let $K \in \R$ and $ N\in[1,  \infty)$. We say that a complete geodesic metric measure space  $(X,\sfd,\mm)$
 satisfies the $\CD^*(K,N)$ condition if for any two measures $\mu_0, \mu_1 \in \prob X$ with support 
bounded and contained in $\supp(\mm)$ there
exists a measure $\ppi \in \gopt(\mu_0,\mu_1)$ such that for every $t \in [0,1]$
and $N' \geq  N$ we have
\begin{equation}\label{eq:CD-def}
-\int\rho_t^{1-\frac1{N'}}\,\d\mm\leq - \int \sigma^{(1-t)}_{K,N'}(\sfd(\gamma_0,\gamma_1))\rho_0^{-\frac1{N'}}+\sigma^{(t)}_{K,N'}(\sfd(\gamma_0,\gamma_1))\rho_1^{-\frac1{N'}}\,\d\ppi(\gamma),
\end{equation}
where for any $t\in[0,1]$ we  have written $(\e_t)_\sharp\ppi=\rho_t\mm+\mu_t^s$  with $\mu_t^s \perp \mm$.

What is different in the $\CD(K,N)$ definition is the choice of the weights $\sigma$. 
In the particular case $K=0$ the $\CD^*(0,N)$ condition is the same as the $\CD(0,N)$ one.

We will in fact only need to show that our example fails the $\CD(K,\infty)$ condition.
For defining the $\CD(K,\infty)$ condition we will need the entropy
\[
 \ent_\infty(\mu) = \int_X\rho\log\rho\,\d\mm,
\]
if $\mu = \rho\mm$ is absolutely continuous with respect to $\mm$ and $\ent_\infty(\mu) = \infty$ otherwise.
We say that a metric measure space  $(X,\sfd,\mm)$
 satisfies the $\CD(K,\infty)$ condition if for any two measures $\mu_0, \mu_1 \in \prob X$ with support 
bounded and contained in $\supp(\mm)$ there
exists a measure $\ppi \in \gopt(\mu_0,\mu_1)$ such that for every $t \in [0,1]$ we have
\[
 \ent_\infty(\mu_t) \le (1-t)\ent_\infty(\mu_0) + t \ent_\infty(\mu_1) - \frac{K}{2}t(1-t)W_2^2(\mu_0,\mu_1),
\]
where we have written $\mu_t = (\e_t)_\sharp\ppi$.

A complete geodesic metric measure space  $(X,\sfd,\mm)$ is said to satisfy $\CD^*(K,N)$ locally if for any $x \in X$
there exists a radius $r>0$ so that for any $\mu_0, \mu_1 \in \prob X$ with supports in $B(x,r)$
there is a measure $\ppi \in \gopt(\mu_0,\mu_1)$ such that for every $t \in [0,1]$
and $N' \geq  N$ we have \eqref{eq:CD-def}.

\subsection{Approximate differentiability and the Jacobian equation}\label{sec:appdif}

Given two absolutely continuous measures $\mu_0, \mu_1 \in \probt{\R^2}$ and an optimal map $T \colon \R^2 \to \R^2$
pushing $\mu_0$ to $\mu_1$,
our aim is to express the density $\rho_1$ of $\mu_1$ using the density $\rho_0$ of $\mu_0$ and the mapping $T$.
Assuming $T$ to be one-to-one and smooth, this expression is the standard Jacobian equation
\begin{equation}\label{eq:jac}
 \rho_1(T(x,y))J_T(x,y) = \rho_0((x,y)) \qquad \text{for }\mu_0\text{-almost every }(x,y),
\end{equation}
where $J_T(x,y)$ is the absolute value of the Jacobian determinant of $T$.
A way to relax the assumptions on $T$ to be one-to-one and smooth is to require it to be one-to-one
almost everywhere and approximately differentiable, see for instance the book by Ambrosio, Gigli and Savar\'e 
\cite[Lemma 5.5.3]{Ambrosio-Gigli-Savare08} for a precise statement.

Recall that a mapping $f \colon U \to \R^m$, $U \subset \R^n$ open, is called \emph{approximately differentiable at} $x \in U$
if there exists a measurable function $\tilde f \colon U \to \R^m$ which is differentiable at $x$ and for which
\[
 \lim_{r \to 0}\frac{\LL_n(\{z \in B(x,r)\,:\,f(z) = \tilde f(z)\})}{\LL_n(B(x,r))} = 1.
\]
The approximate differential of $f$ at $x$ is defined to be that of $\tilde f$ at $x$. Correspondingly we define
the approximate partial derivatives (of the components), denoted simply by $\frac{\partial f_i}{\partial z_i}$.

Approximate differentiability for $T$ would follow from the almost everywhere existence of approximate
partial derivatives, see Federer's book \cite[Theorem 3.1.4]{Federer}. However, our mapping will not in general have approximate partial derivatives in all the directions.
Due to the special structure of our optimal maps the following easy version will suffice.
In the proposition below, and later on, we write the components of a map $f \colon \R^2 \to \R^2$ 
as $f_1$ and $f_2$. In other words $f(x,y) = (f_1(x,y),f_2(x,y))$.

\begin{proposition}\label{prop:appjac}
 Let $\mu_0, \mu_1 \in \probt{\R^2}$ be absolutely continuous with respect to $\LL_2$ with densities $\rho_0$ and $\rho_1$,
 respectively, and let $f \colon \R^2 \to \R^2$ be a map
 such that $\mu_1 = f_\sharp \mu_0$ and $f$ is one-to-one outside a set of measure zero.
 Suppose that $f_1(x,y) = f_1(x)$, i.e. $f_1$ does not depend on $y$.
 Suppose also that $f_1$ is increasing in $x$ and that $f_2(x,y)$ is increasing in $y$ for almost every $x \in \R$.
 Then \eqref{eq:jac} holds with $J_f(x,y)=\frac{\partial f_1}{\partial x}\frac{\partial f_2}{\partial y}$.
\end{proposition}
\begin{proof}
 Because $f_1(x)$ is increasing in $x$ and $f_2(x,y)$ is increasing in $y$ for almost every $x$,
 $f_1$ is almost everywhere approximately differentiable and $f_2$ has an approximate partial derivative
 in the $y$-direction at almost every point.

 Take a measurable $A \subset \R^2$ and write $A_x = \{y \in \R\,:\, (x,y) \in A\}$. Since
 $\mu_0$ and $\mu_1$ are absolutely continuous with respect to $\LL_2$ and $f$ is one-to-one
 outside a set of measure zero, we have
 \begin{align*}
   \int_A \rho_0((\bar x, \bar y)) \,\d \LL_2(\bar x ,\bar y) = \mu_0(A)
   &  = \mu_1(f(A)) = \int_{f(A)} \rho_1(\tilde x, \tilde y)\,\d\tilde y\,\d\tilde x \\
  & = \int_{-\infty}^\infty \int_{A_{\bar x}} \rho_1(f(\bar x,\bar y))\frac{\partial f_2}{\partial y}(\bar x,\bar y)\,\d \bar y
                \frac{\partial f_1}{\partial x}(\bar x) \,\d \bar x\\
  & = \int_A \rho_1(f(\bar x,\bar y))\frac{\partial f_2}{\partial y}(\bar x,\bar y)
                \frac{\partial f_1}{\partial x}(\bar x, \bar y) \,\d \LL_2(\bar x, \bar y).
 \end{align*}
 The claim follows from this.
\end{proof}

\section{Details of the example}\label{sec:verify}

Most of this section is devoted to verifying the local $\CD(0,4)$ condition in Example \ref{ex:1}. The plan is to use the Jacobian equation
to estimate the density along a chosen geodesic in $\probt{\R^2}$. Before arriving at this we will
first show that we have an optimal map $T$ between two absolutely continuous measures $\mu_0$ and $\mu_1$,
that this map is essentially one-to-one and that it can be used in a Jacobian equation. Using the optimal map $T$
we will then select a midpoint measure whose support is still inside our domain. Here we also have to make sure that the
map sending an initial point to the midpoint is essentially one-to-one. Finally we will verify that this midpoint measure
satisfies $\CD(0,4)$. At the very end we will also indicate why the global $\CD(K,\infty)$ condition fails.

\subsection{Definition of the local domain}

Since Theorem \ref{thm:CD02maps} is proven by approximating the norm $||\cdot||_\infty$ with
strictly convex norms, the $\CD(0,4)$ condition (in fact the $\CD(0,2)$ condition) holds inside any domain that is convex in the Euclidean sense.
What needs to be done is to verify the $\CD(0,4)$ condition inside domains of the type shown in Figure \ref{fig:2}.

\begin{figure}
   \psfrag{x}{$(x,y)$}
   \psfrag{y}{$(x,S(x))$}
   \psfrag{e}{$d + \frac{b-a}2$}
   \psfrag{a}{$a$}
   \psfrag{b}{$b$}
   \psfrag{c}{$c$}
   \centering
   \includegraphics[width=0.5\textwidth]{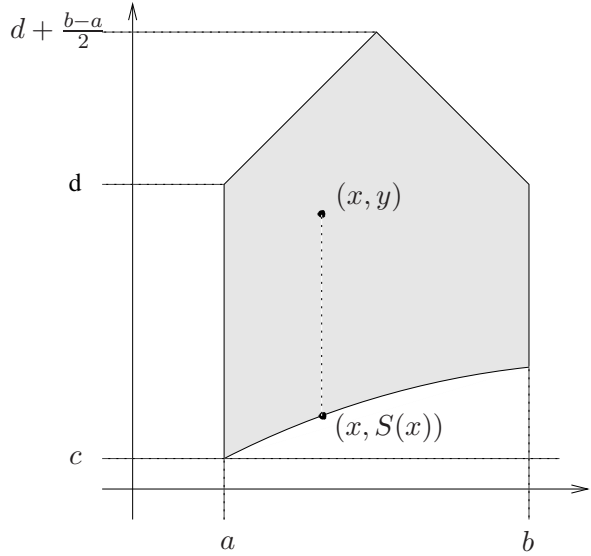}
   \caption{The local domain $E$ where we verify the $\CD(0,4)$ condition.}
   \label{fig:2}
\end{figure}

Referring to Figure \ref{fig:2} for the notation,
the width $b-a$ and the height $d + \frac{b-a}2 - c$ of the domain $E$ are assumed to be less than $\frac1{64}$.
The bottom of the domain
is a piece of a sphere with radius one and whose center $(x_c,y_c)$ satisfies
\begin{equation}\label{eq:spherecenter}
 |x - x_c| < \frac12(y - y_c) \qquad \text{for all }(x,y) \in E. 
\end{equation}

Let for every $(x,y) \in [a,b] \times \R$
the point $(x,S(x))$ be the vertical projection to the lower (circular) boundary 
of $E$, see Figure \ref{fig:2}. Notice that $S(x)$ only depends on $x$. Our assumption \eqref{eq:spherecenter} guarantees
\begin{equation}\label{eq:smallverticaldifference}
 S\left(\frac{x_1+x_2}{2}\right) - \frac{S(x_1)+S(x_2)}{2} \le \frac{|x_1-x_0|^2}{2} \qquad \text{for all } x_1,x_2 \in [a,b]. 
\end{equation}

\subsection{Preliminary reductions and definitions}

Let us now mention two simplifications that we can always make when checking the $\CD^*(K,N)$ condition. We will return to both of them
in more detail at the end of the paper when we finally prove the $\CD(0,4)$ condition.
The first standard reduction in checking the $\CD^*(K,N)$ condition is to assume the measures to be absolutely continuous
with respect to the reference measure. This reduction is possible because we can approximate any probability measure
in the $W_2$-distance by an absolutely continuous measure without increasing the entropy.

The second standard simplification we make is that we only define the midpoint between any two given measures.
This has been used for example by Bacher and Sturm \cite{BacherSturm10} and the author \cite{R2013}.
We can then iterate the procedure of taking midpoints and use the lower semi-continuity of the entropy to have the
correct entropy bound along the whole geodesic.

Let us then turn to the notation and definitions that are less standard than the ones we recalled in Section \ref{sec:preli}.
Given a metric space $(X,\sfd)$, for $z_0,z_1 \in X$ we denote the set of all the midpoints between $z_0$ and $z_1$ by
\[
 \midp(z_0,z_1) := \left\{z\in X \,:\, \sfd(z_0,z) = \sfd(z_1,z) = \frac12 \sfd(z_0,z_1)\right\}.
\]
We will not make the distance $\sfd$ visible in the notation because $\midp$ will only be used for $(\R^2, ||\cdot||_\infty)$
and $(\probt{\R^2}, W_2)$, and for those no confusion should arise.

In the following we will often consider separately the part of the transport that moves more in the horizontal (or vertical) direction.
To set some notation define the set of horizontal transportation
\[
 H := \left\{((x_0,y_0),(x_1,y_1)) \in \R^2 \times \R^2 \,:\,|x_0 -x_1| > |y_0 - y_1|\right\},
\]
the set of vertical transportation
\[
 V := \left\{((x_0,y_0),(x_1,y_1)) \in \R^2 \times \R^2 \,:\,|x_0 -x_1| < |y_0 - y_1|\right\}
\]
and the set of diagonal transportation
\[
 D := \left\{((x_0,y_0),(x_1,y_1)) \in \R^2 \times \R^2 \,:\,|x_0 -x_1| = |y_0 - y_1|\right\}.
\]
Given any $\ggamma \in \prob{\R^2 \times \R^2}$,
the restricted measure $\ggamma\restr{H}$ moves every infinitesimal mass more in the horizontal direction than the vertical,
$\ggamma\restr{V}$ the other way around, and $\ggamma\restr{D}$ moves mass in the diagonal directions.

\subsection{Selecting an optimal map}

One possible way of trying to obtain the needed optimal maps could be to analyse the proof of Theorem \ref{thm:CD02maps},
or the $\CD(0,2)$ condition in $(\R^2,||\cdot||_\infty,\LL_2)$. 
However, we chose a more direct approach of first selecting a suitable optimal transport plan via three consecutive minimizations
and then showing that this plan has all the desired properties.
The idea behind the three minimizations is that the $l^\infty$-norm allows locally a lot of freedom for the coordinate
in which the mass is transported less.  By doing extra minimization on the two directions separately after the main minimization,
we will increase the monotonicity properties of the optimal transport. 

The idea of using consecutive minimizations to choose a better optimal transport plan goes back to
\cite{AmbrosioKirchheimPratelli, Champion-DePascale11} where the existence of optimal maps from absolutely continuous 
measures in $\R^n$ for cost functions of the form $c(x,y) = ||x-y||$ was proven - first with
any crystalline norm $||\cdot||$ by Ambrosio, Kirchheim and Pratelli in \cite{AmbrosioKirchheimPratelli}
and then with any norm $||\cdot||$ by Champion and De Pascale in \cite{Champion-DePascale11}.
Let us also note that the existence of an optimal map in our case with $c(x,y) = ||x-y||_\infty^2$ has been proven
by Carlier, De Pascale and Santambrogio in \cite{CarlierDePascaleSantambrogio}.
We will prove here the existence of a specific optimal transport map using the consecutive minimizations in order to keep the paper
reasonably self-contained and, more importantly, in order to guarantee that the chosen optimal plan has all the needed cyclical 
monotonicity properties.


Let us give the three minimizations. Suppose that $\mu_0, \mu_1 \in \probt{\R^2}$ are given.
Let $\Op_1(\mu_0,\mu_1) \subset \prob{\R^2 \times \R^2}$ be the set of those $\gamma$ that minimize
\begin{equation}\label{eq:firstmin}
  \int_{\R^2 \times \R^2} ||z_1-z_2||^2 \d\ggamma(z_1,z_2)
\end{equation}
and satisfy $(\ppi_1)_\sharp\ggamma = \mu_0$ and $(\ppi_2)_\sharp\ggamma = \mu_1$.
The set $\Op_1(\mu_0,\mu_1)$ is a nonempty closed and convex subset of $\prob{\R^2 \times \R^2}$.
Next let $\Op_2(\mu_0,\mu_1)  \subset \Op_1(\mu_0,\mu_1)$ be the set of those $\gamma$ that minimize
\begin{equation}\label{eq:secondmin}
  \int_{\R^2 \times \R^2} |x_1-x_2|^2 \d\ggamma((x_1,y_1),(x_2,y_2)).
\end{equation}
Again $\Op_2(\mu_0,\mu_1)$ is a nonempty closed and convex subset of $\prob{\R^2 \times \R^2}$.
Finally let $\Op_3(\mu_0,\mu_1) \subset \Op_2(\mu_0,\mu_1)$ be the set of those $\gamma$ that minimize
\begin{equation}\label{eq:finalmin}
  \int_{\R^2 \times \R^2} |y_1-y_2|^2 \d\ggamma((x_1,y_1),(x_2,y_2)).
\end{equation}
Clearly also $\Op_3(\mu_0,\mu_1)$ is nonempty. We will see in Proposition \ref{prop:exmap} that in the case
$\mu_0 \ll \LL_2$ the set  $\Op_3(\mu_0,\mu_1)$ consists of only one optimal plan which is given by a map. Before this, let us list the
cyclical monotonicity properties we immediately get from the three minimizations.

\begin{lemma}\label{lma:cycmon}
 Let $\mu_0,\mu_1 \in \probt \R^2$ and $\ggamma \in \Op_3(\mu_0,\mu_1)$. Then there exists a set $\Gamma \subset \R^2 \times \R^2$
 of full $\ggamma$-measure such that for all $(z_1,w_1), (z_2,w_2) \in \Gamma$ we have
 \begin{equation}\label{eq:comcyc}
   ||z_1 - w_1||^2 + ||z_2 - w_2||^2 \le ||z_1 - w_2||^2 + ||z_2 - w_1||^2
 \end{equation}
 and for all $((x_1,y_1),(x_1',y_1')), ((x_2,y_2),(x_2',y_2')) \in \Gamma$ we have
 \begin{equation}\label{eq:vercyc}
  \begin{split}|y_1 - y_1'|^2 + |y_2 - y_2'|^2 \le |y_1 - y_2'|^2 + |y_2 - y_1'|^2,\\
           \text{if } |x_1-x_1'|^2 + |x_2-x_2'|^2 = |x_2-x_1'|^2 + |x_1-x_2'|^2  
  \end{split} 
 \end{equation}
 and 
 \begin{equation}\label{eq:horcyc}
 \begin{split}
  |x_1 - x_1'|^2 + |x_2 - x_2'|^2 \le |x_1 - x_2'|^2 + |x_2 - x_1'|^2,\\
        \text{if } |y_1-y_1'|^2 + |y_2-y_2'|^2 = |y_2-y_1'|^2 + |y_1-y_2'|^2.
 \end{split}
 \end{equation}
\end{lemma}

Let us then prove that in the case $\mu_0 \ll \LL_2$ the optimal plan in $\Op_3(\mu_0,\mu_1)$ is given by a map.
This is a fairly standard consequence of Lemma \ref{lma:cycmon}, so we present only parts of the proof to give the idea.

\begin{proposition}\label{prop:exmap}
 Suppose $\mu_0 \ll \LL_2$. Then $\Op_3(\mu_0,\mu_1)$ is a singleton and its only element is induced by an optimal map $T$.
\end{proposition}
\begin{proof}
 The fact that $\Op_3(\mu_0,\mu_1)$ is a singleton follows once we know that any element in $\Op_3(\mu_0,\mu_1)$ is 
 induced by an optimal map. Indeed, if there were two different measures $\ggamma_1, \ggamma_2 \in \Op_3(\mu_0,\mu_1)$,
 then by convexity also $\ggamma_3 = \frac12(\ggamma_1 + \ggamma_2) \in \Op_3(\mu_0,\mu_1)$. However, the measure $\ggamma_3$
 would not be induced by a map.

 Suppose now that there exists $\ggamma \in \Op_3(\mu_0,\mu_1)$ that is not induced by a map. Then the disintegration $\ggamma_z$ of
 $\ggamma$ with respect to $\ppi_1$ is not a Dirac mass for a $\mu$-positive set of points $z \in \R^2$.
 Now there are several cases to check. We use different cyclical monotonicities to arrive at a contradiction in each of the cases.
 The different cases are:
 \begin{enumerate}
  \item[(i)] $\ggamma_z(H) > 0$ and $\ggamma_z(V)>0$ for a $\mu_0$-positive set of $z$.
  \item[(ii)] $\gamma_z\restr{H}$, $\gamma_z\restr{V}$ or $\gamma_z\restr{D}$ is not a Dirac mass for a $\mu_0$-positive set of $z$.
  \item[(iii)] $\gamma_z(D) > 0$ and $\ggamma_z(H) > 0$ (or $\ggamma_z(V)>0$) for a $\mu_0$-positive set of $z$.
 \end{enumerate}
 The contradiction follows from all of the cases in a similar way.
 We will only give details in the first case. Thus assume that $\ggamma_z(H) > 0$ and $\ggamma_z(V)>0$ for a $\mu_0$-positive set of $z$.
 Let $\Gamma \subset \R^2 \times \R^2$ be the set from Lemma \ref{lma:cycmon} having all the cyclical monotonicity properties.
 Suppose that the set
 \[
   \{z \in \R^2\,:\,\ggamma_z(H\cap\Gamma) > 0 \text{ and }\ggamma_z(V\cap\Gamma)>0\}
 \]
 has positive $\mu_0$-measure. Now there exist $\epsilon,\delta > 0$ and $(x_h,y_h), (x_v,y_v) \in \R^2$ so that
 \[
  ||(x_h,y_h) - (x_v,y_v)|| \ge 4\delta  
 \]
 and the set
 \begin{align*}
 A = \big\{z \in& \R^2\,:\,\ggamma_z\left(\{((x_1,y_1),(x_2,y_2)) \,:\, |y_1-y_2| < |x_1 - x_2| - \epsilon\}\cap\Gamma\cap \R^2 \times B((x_h,y_h),\delta)\right) > 0\\
                         &\text{ and } \ggamma_z\left(\{((x_1,y_1),(x_2,y_2)) \,:\, |x_1-x_2| < |y_1 - y_2| - \epsilon\}\cap\Gamma\cap \R^2 \times B((x_v,y_v),\delta)\right)>0\big\}  
 \end{align*}
 has positive $\mu_0$-measure. Let $(\bar x, \bar y)$ be a density point of $A$.
 By symmetry, assume $||(\bar x, \bar y) - (x_v,y_v)|| \ge 2 \delta$.
 Because $(\bar x, \bar y)$ is a density point, for some $x \in [\bar x - \frac\epsilon2 ,\bar x + \frac\epsilon2]$ 
 there exist $y_1,y_2 \in [\bar y - \frac\epsilon2 ,\bar y + \frac\epsilon2]$, $y_1 \ne y_2$,
 such that $(x,y_1), (x,y_2) \in A$. We may assume that $|y_v - y_2| < |y_v - y_1|$.
 Let $(x_{h,2},y_{h,2}) \in B((x_h,y_h),\delta)$ and $(x_{v,1},y_{v,1}) \in B((x_v,y_v),\delta)$
 be such that
 \[
  ((x,y_1),(x_{v,1},y_{v,1})), ((x,y_2),(x_{h,2},y_{h,2})) \in \Gamma,
 \]
 \[
  |y_2-y_{h,2}| < |x - x_{h,2}| - \epsilon \quad \text{and} \quad |x-x_{v,1}| < |y_1 - y_{v,1}| - \epsilon .
 \]
 But now
 \begin{align*}
  ||(x,y_2) -(x_{v,1},y_{v,1})||^2 & + ||(x,y_1) - (x_{h,2},y_{h,2})||^2\\
& = |y_2-y_{v,1}|^2 + |x-x_{h_2}|^2 < |y_1-y_{v,1}|^2 + |x-x_{h_2}|^2 \\
& = ||(x,y_1) -(x_{v,1},y_{v,1})||^2 + ||(x,y_2) - (x_{h,2},y_{h,2})||^2
 \end{align*}
 contradicting the cyclical monotonicity \eqref{eq:comcyc} of $\Gamma$.
 This proves the claim in the case (i).

 In the case (ii) we argue similarly and use the cyclical monotonicities \eqref{eq:comcyc} and \eqref{eq:vercyc}
 if  $\gamma_z\restr{H}$ is not Dirac, \eqref{eq:comcyc} and \eqref{eq:horcyc} if  $\gamma_z\restr{V}$ is not,
 and \eqref{eq:comcyc} if $\gamma_z\restr{D}$ is not.
 In the case (iii) we use \eqref{eq:vercyc} if $\gamma_z(D) > 0$ and $\ggamma_z(H) > 0$,
 and \eqref{eq:horcyc} if $\gamma_z(D) > 0$ and $\ggamma_z(V) > 0$.
\end{proof}

Next we list some properties of the map $T$ in the case $\mu_0, \mu_1 \ll \LL_2$.

\begin{lemma}\label{lma:monotone}
 Let $\mu_0 \ll \LL_2$, $T$ the map from Proposition \ref{prop:exmap} and $\Gamma$ the set from Lemma \ref{lma:cycmon}.
 Then for all $(x,y_1),(x,y_2),(x_1,y),(x_2,y) \in \{(x,y) \in \R^2 \,:\, ((x,y),T(x,y)) \in \Gamma\}$ we have the following.
 \begin{equation}\label{eq:correctorder1}
  \text{If }y_1 \ne y_2 \text{ and }T_1(x,y_1) = T_1(x,y_2), \text{ then }\frac{T_2(x,y_1) - T_2(x,y_2)}{y_1 - y_2} \ge 0
 \end{equation}
 and 
\begin{equation}\label{eq:correctorder2}
  \text{if }x_1 \ne x_2 \text{ and }T_2(x_1,y) = T_2(x_2,y), \text{ then }\frac{T_1(x_1,y) - T_1(x_2,y)}{x_1 - x_2} \ge 0.
 \end{equation}
\end{lemma}
\begin{proof}
 Suppose that \eqref{eq:correctorder1} does not hold for some $(x,y_1),(x,y_2) \in F$.
 We may assume that $y_2 < y_1$ so that $T_2(x,y_1) < T_2(x,y_2)$. By the cyclical monotonicity
 \eqref{eq:comcyc} we have $|T_1(x,y_1)-x| \ge |T_2(x,y_1) - y_1|$ and $|T_1(x,y_2)-x| \ge |T_2(x,y_2) - y_2|$.
 Therefore
 \[
  ||T(x,y_1)-(x,y_1)|| = ||T(x,y_1)-(x,y_2)|| = ||T(x,y_2)-(x,y_2)|| = ||T(x,y_2)-(x,y_1)||.
 \]
 But now
 \[
  |T_2(x,y_1)-y_2|^2 + |T_2(x,y_2) - y_1|^2 < |T_2(x,y_1)-y_1|^2 + |T_2(x,y_2) - y_2|^2
 \]
 violating the cyclical monotonicity of \eqref{eq:vercyc}. This proves \eqref{eq:correctorder1}.
 The inequality \eqref{eq:correctorder2} follows similarly from the cyclical monotonicities \eqref{eq:comcyc}
 and \eqref{eq:horcyc}.
\end{proof}

In estimating the densities at the midpoints we will also need an infinitesimal version of Lemma \ref{lma:monotone}.
Recalling the discussion from Section \ref{sec:appdif} we would like to use a Jacobian equation
\begin{equation}\label{eq:jacagain}
  \rho_1(T(x,y))J_T(x,y) = \rho_0((x,y)) \qquad \text{for }\mu_0\text{-almost every }(x,y). 
\end{equation}

Here a few comments are in order. As we mentioned in Section \ref{sec:appdif}, usually in writing the Jacobian equation 
the mapping is assumed to be at least approximately differentiable almost everywhere.
However, the optimal map $T$ is not in general approximately differentiable.
To see this, take a measurable function $f \colon [0,1] \to [0,1]$ that is not approximately differentiable
and consider the optimal transport between the uniform measures on $[0,1]^2$ and $\{(x+3,y)\,:\,x \in [0,1],y \in [f(x),f(x)+1]\}$.

Nevertheless, because locally in $H$ we are sending vertical lines to vertical lines by cyclical monotonicity \eqref{eq:comcyc}
the first coordinate function $T_1$ is approximately differentiable almost everywhere.
Then, because of cyclical monotonicity \eqref{eq:vercyc}
the second coordinate function $T_2$ is approximately differentiable in the variable $y$ for almost every $x$.
Now, since $T_1$ was locally (approximately) constant in $y$, we get \eqref{eq:jacagain} in $H$ using Proposition \ref{prop:appjac}.
Similarly we get it also in $V$ and $D$.

\begin{lemma}\label{lma:partials}
 Let $\mu_0, \mu_1 \ll \LL_2$ and $\Op_3(\mu_0,\mu_1) = \{(\id,T)_\sharp\}$. Then the map $T$
 satisfies $\mu_0$-almost everywhere
 \begin{equation}\label{eq:nonnegderivative}
  \frac{\partial T_1}{\partial x} \ge 0 \text{ and } \frac{\partial T_2}{\partial y} \ge 0, \text{ if }((x,y),T(x,y)) \in H \cup V.
 \end{equation}
 Still $\mu_0$-almost everywhere we have that 
 \begin{equation}\label{eq:mixedzero}
  \begin{split}
    T_1 & \text{ is locally constant in }y, \text{ if }((x,y),T(x,y)) \in H\text{ and}\\
    T_2 & \text{ is locally constant in }x, \text{ if }((x,y),T(x,y)) \in V.
   \end{split}
 \end{equation}
\end{lemma}
\begin{proof}
 In proving \eqref{eq:nonnegderivative} assume first that $((x,y),T(x,y)) \in H$. Then by the cyclical monotonicity \eqref{eq:comcyc}
 we have $\frac{\partial T_1}{\partial x} \ge 0$. Notice that in $H$ vertical lines are locally sent to vertical 
 lines so that $\frac{\partial T_2}{\partial y} \ge 0$ follows from \eqref{eq:correctorder1}.
 In a similar way we can prove \eqref{eq:nonnegderivative} assuming $((x,y),T(x,y)) \in V$.

 The first claim in \eqref{eq:mixedzero} follows again from the observation that in $H$ vertical lines are locally sent to vertical 
 lines, and the second claim follows analogously.
\end{proof}

\subsection{Defining the midpoint}

As we already saw in the Introduction (Figure \ref{fig:idea}) we have to deviate the midpoint of a geodesic
from the Euclidean midpoint by an amount depending on the endpoints of the geodesics.
A geodesic going in the 45 degree direction has to remain
the same geodesic and a geodesic going in the horizontal direction can deviate the most.

The idea behind defining the midpoint the way we do here is that we want to keep the height of the transport
right for a (vertical) $\CD(0,2)$ condition. If the height is exactly the correct one for the condition 
between vertical strips with their base on the sphere bounding our domain, it will also be infinitesimally correct.

Naturally the correction for the midpoints needs to be done only in the horizontal part $H$ of the transport. For the vertical part $V$ and 
the diagonal part $D$ we can use the Euclidean midpoints (who will respectively give a $\CD(0,2)$ and $\CD(0,1)$ condition 
for those parts of the transport).

The midpoint $\mu_\frac12$ will be defined using the mapping $M \colon E \times E \to \R^2$ given by
\begin{equation}\label{eq:Mdef1}
  M\left((x_0,y_0),(x_1,y_1)\right) =  \left(\frac{x_0+x_1}2,\frac{y_0+y_1}2\right),
\end{equation}
if $((x_0,y_0),(x_1,y_1)) \notin H$ (corresponding to the Euclidean midpoint in the non horizontal transport), and by
\begin{equation}\label{eq:Mdef2}
 \begin{split}
  M\left((x_0,y_0),(x_1,y_1)\right) =  \bigg(\frac{x_0+x_1}2, & \frac{S(x_0)+  S(x_1)}2 + (x_0-x_1)^2\\
                                      & + \frac14 \left(\sqrt{y_0 - S(x_0)}+\sqrt{y_1 - S(x_1)}\right)^2 \bigg),
 \end{split}
\end{equation} 
if $((x_0,y_0),(x_1,y_1)) \in H$ (corresponding to the vertical shrinking to satisfy the $CD(0,2)$ condition in the horizontal transport).

The first thing to check is that $M$ really gives midpoints. As usual, we write $M = (M_1,M_2)$.

\begin{lemma}\label{lma:midpoint}
  $M\left((x_0,y_0),(x_1,y_0)\right) \in \midp\left((x_0,y_0),(x_1,y_0)\right)$.
\end{lemma}
\begin{proof}
 We may assume $x_0 \le x_1$.
 If $((x_0,y_0),(x_1,y_1)) \notin H$, the claim is obvious. Let then $((x_0,y_0),(x_1,y_1)) \in H$ so that $M$ is given by \eqref{eq:Mdef2}.
 By symmetry, we may assume that $y_0 \le y_1$. We have to show that
 \begin{equation}\label{eq:midpoint1}
  M_2\left((x_0,y_0),(x_1,y_1)\right) - y_0 \le \frac{x_1-x_0}2
 \end{equation}
 and
 \begin{equation}\label{eq:midpoint2}
  y_1 - M_2\left((x_0,y_0),(x_1,y_1)\right) \le  \frac{x_1-x_0}2.
 \end{equation}
 Because $M_2$ is increasing in $y_1$, for verifying \eqref{eq:midpoint1} it is enough to check the extreme case
 $y_1 - y_0 = x_1 - x_0$ (even though in this case the mapping $M$ is defined using \eqref{eq:Mdef1}).
 Notice that by assumption on the width and height of $E$ we have
 \begin{equation}\label{eq:164}
  |y_1-y_0|,|y_1-S(x_1)| \le \frac1{64}
 \end{equation}
 and by \eqref{eq:spherecenter} we have
 \begin{equation}\label{eq:smallangle}
  |S(x_1) - S(x_0)| \le \frac12|x_1 - x_0| = \frac12 |y_1 - y_0|.
 \end{equation}
 Together \eqref{eq:164} and \eqref{eq:smallangle} yield
 \[
  (|y_1-y_0| - |S(x_1) - S(x_0)|)^2 \ge \frac14 |y_1-y_0|^2 \ge 8(|y_1 - S(x_1)| + |y_1-y_0|)|y_1-y_0|^2.
 \]
 This immediately gives
 \[
  4(y_1-y_0)^2 \le \left(\sqrt{y_0-S(x_0)} - \sqrt{y_1 - S(x_1)}\right)^2,
 \]
 which is \eqref{eq:midpoint1} in the extreme case $y_1 - y_0 = x_1 - x_0$.
 
 In checking \eqref{eq:midpoint2} we can use the fact that $M_2$ is increasing in $y_0$. Hence we only need to check
 the extreme case $y_0 = y_1$. Again by symmetry we may assume $S(x_0) \le S(x_1)$. Because of
 \eqref{eq:smallangle} we have
 \[
  2y_0 - S(x_0) - S(x_1) - 2|x_1-x_0| \le 2(y_0 -S(x_1)) \le 2\sqrt{(y_0-S(x_1))(y_0 - S(x_0))}.
 \]
 Therefore
 \[
  \frac14\left(\sqrt{y_0 - S(x_0)} - \sqrt{y_0 - S(x_1)}\right)^2 \le \frac{x_1 - x_0}{2} \le \frac{x_1 - x_0}{2} + (x_1 - x_0)^2,
 \]
 which is the inequality \eqref{eq:midpoint2} in the critical case.
\end{proof}

The second thing to check is that the midpoints are inside our domain $E$.

\begin{lemma}
 The mapping $M$ has values in $E$.
\end{lemma}
\begin{proof}
 Again, if $((x_0,y_0),(x_1,y_1)) \notin H$, the claim is obvious. Hence, suppose $((x_0,y_0),(x_1,y_1)) \in H$.
 By Lemma \ref{lma:midpoint} we know that $M\left((x_0,y_0),(x_1,y_0)\right) \in \midp\left((x_0,y_0),(x_1,y_0)\right)$.
 Therefore the only thing to check is that
 \[
  M\left((x_0,y_0),(x_1,y_0)\right) > S\left(\frac{x_0 + x_1}{2}\right).
 \]
 This follows from our assumptions on the domain $E$, more precisely from \eqref{eq:smallverticaldifference}.
\end{proof}

\subsection{Verifying the local $\CD(0,4)$ condition}

In order to be able to use the Jacobian equation \eqref{eq:jac} for the midpoints we first have to 
check that our mapping giving the midpoint is essentially one-to-one.

\begin{lemma}\label{lma:onetoone}
 Let $\mu_0,\mu_1 \in \prob{E}$ with $\mu_0, \mu_1 \ll \LL_2$. Let $T$ be the optimal map from Proposition \ref{prop:exmap}.
 Then the map $M \circ (\id, T)$ is one-to-one outside a set of $\mu_0$-measure zero.
\end{lemma}
\begin{proof}
 Let $\Gamma$ be the set from Lemma \ref{lma:cycmon}.
 Suppose that there exist $(x_1,y_1), (x_2,y_2) \in E$ so that $((x_1,y_1),T(x_1,y_1)), ((x_2,y_2),T(x_2,y_2)) \in \Gamma$,
  $(x_1,y_1) \ne (x_2,y_2)$ and
 \begin{equation}\label{eq:equal}
   M \circ (\id, T)(x_1,y_1) = M \circ (\id, T)(x_2,y_2).  
 \end{equation}
 We have three cases to check:
 \begin{enumerate}
  \item[(i)] $((x_1,y_1),T(x_1,y_1)), ((x_2,y_2),T(x_2,y_2)) \in H$
  \item[(ii)] $((x_1,y_1),T(x_1,y_1)), ((x_2,y_2),T(x_2,y_2)) \notin H$
  \item[(iii)] $((x_1,y_1),T(x_1,y_1)) \in H$, $((x_2,y_2),T(x_2,y_2)) \notin H$.
 \end{enumerate}
 In the case (i) we may assume $x_1 = x_2$ and $T_1(x_1,y_1) = T_1(x_2,y_2)$ by cyclical monotonicity \eqref{eq:comcyc}
 and $y_1 < y_2$ by symmetry. Then by Lemma \ref{lma:monotone} we have $T_2(x_1,y_1) < T_2(x_2,y_2)$.
 Since $M_2$ is strictly increasing in both of the $y$-coordinates, we have 
 \[
  M_2 \circ (\id, T)(x_1,y_1) < M_2 \circ (\id, T)(x_2,y_2)  
 \]
 contradicting the assumption \eqref{eq:equal}.
 
 In the case (ii) we may first of all assume $y_1 = y_2$ and $T_2(x_1,y_1) = T_2(x_2,y_2)$ by cyclical monotonicity \eqref{eq:comcyc}.
 The assumption \eqref{eq:equal} gives
 \[
  \frac{x_1 + T_1(x_1,y_1)}{2} = \frac{x_2 + T_1(x_2,y_2)}{2}.
 \]
 This implies via Lemma \ref{lma:monotone} that $x_1 = x_2$, which contradicts the assumption $(x_1,y_1) \ne (x_2,y_2)$.

 Finally we have the case (iii). We may assume $x_1 < T_1(x_1,y_1)$. If $T_1(x_2,y_2) < T_1(x_1,y_1)$,
 then
 \[
  ||(x_1,y_1) - T(x_2,y_2)|| < ||(x_1,y_1) - T(x_1,y_1)|| = ||(x_2,y_2) - T(x_2,y_2)||
 \]
 contradicting the cyclical monotonicity \eqref{eq:comcyc}. On the other hand, if $T_1(x_2,y_2) = T_1(x_1,y_1)$,
 we have $y_2 < y_1$, $x_1 = x_2$ and $T_2(x_1,y_1) < T_2(x_2,y_2)$ contradicting \eqref{eq:correctorder1}.
\end{proof}

Now we are in a position to estimate the density of the midpoint measure.

\begin{proposition}\label{prop:densityalonggeodesics}
 Let $\mu_0,\mu_1 \in \prob{E}$ with $\mu_0, \mu_1 \ll \LL_2$. 
 Then for all $N \ge 4$ we have
 \[
  \ent_N(\mu_\frac12) \le \frac12\left(\ent_N(\mu_0) + \ent_N(\mu_1)\right),
 \]
 where $\mu_\frac12 = (M \circ (\id,T))_\sharp\mu_0$ with $T$ being the optimal map from Proposition \ref{prop:exmap}.
\end{proposition}
\begin{proof}
 We will show that for $\mu_0$-almost every $(x,y) \in E$ we have
 \begin{equation}\label{eq:goalalonggeodesics}
   \rho_\frac12(M((x,y),T(x,y))^{-\frac14} \ge \frac12\left(\rho_0((x,y))^{-\frac14} + \rho_1(T(x,y))^{-\frac14}\right),
 \end{equation}
 where $\mu_0 = \rho_0\LL_2$, $\mu_1 = \rho_1\LL_2$ and $\mu_\frac12 = (M\circ(\id,T))_\sharp \mu_0 = \rho_\frac12\LL_2$.
 The claim of the Proposition then follows by H\"older's inequality and integration.

 By Lemma \ref{lma:onetoone} the mapping $M\circ(\id,T)$ is essentially one-to-one. 
 Our claim \eqref{eq:goalalonggeodesics} will therefore follow if we are able to show that 
 \begin{equation}\label{eq:goal2}
  J_{M \circ (\id,T)}(x,y)^{\frac14} \ge \frac12\left(1 + J_T(x,y)^{\frac14}\right)
 \end{equation}
 holds $\mu_0$-almost everywhere.

 By Lemma \ref{lma:partials} we have $\mu_0$-almost everywhere in $H \cup V$ that
 $T_1$ is locally constant in $y$, $\frac{\partial T_1}{\partial x} \ge 0$ 
 and $\frac{\partial T_2}{\partial y} \ge 0$.
 Thus $\mu_0$-almost everywhere in $H \cup V$ we can write, using Proposition \ref{prop:appjac},
 \begin{equation}
  J_T(x,y) = \frac{\partial T_1}{\partial x}\frac{\partial T_2}{\partial y}.
 \end{equation}
 For the density $\rho_\frac12$ we will need to estimate the Jacobian determinant of the mapping $M \circ (\id,T)$ which is given by
 \[
  (M \circ (\id,T))(x,y) = \left(\frac{x+T_1}2, \frac{y+T_2}2\right),
 \]
 if $((x,y),T(x,y)) \notin H$, and by
 \[
  (M \circ (\id,T))(x,y) = \left(\frac{x+T_1}2, \frac{S(x) + S(T_1)}2 + \frac14 \left(\sqrt{y -S(x)}+\sqrt{T_2 - S(T_1)}\right)^2 + (x-T_1)^2\right),
 \]
 if $((x,y),T(x,y)) \in H$.

 Again by Lemma \ref{lma:partials} we have $\mu_0$-almost everywhere
 \[
  J_{M \circ (\id,T)}(x,y) = \frac12\left(1 + \frac{\partial T_1}{\partial x}\right) \cdot\frac12\left(1 + \frac{\partial T_2}{\partial y}\right),
 \]
 if $((x,y),T(x,y)) \in V$, and 
 \[
   J_{M \circ (\id,T)}(x,y)
   = 
   \frac12\left(1 + \frac{\partial T_1}{\partial x}\right) \cdot \frac14 \left(1 + \frac{\partial T_2}{\partial y}
    + \sqrt{\frac{T_2 - S(T_1)}{y - S(x)}} + \sqrt{\frac{y - S(x)}{T_2 - S(T_1)}}\frac{\partial T_2}{\partial y} \right),
 \]
 if $((x,y),T(x,y)) \in H$.
 
 Let us check \eqref{eq:goal2} in the case $((x,y),T(x,y)) \in H$. The case $((x,y),T(x,y)) \in V$ follows easily
 and the case $((x,y),T(x,y)) \in D$ will be considered at the end of the proof. First observe that
 \[
  1 + \frac{\partial T_1}{\partial x} \ge \frac12\left(1 + \sqrt{\frac{\partial T_1}{\partial x}}\right)^2
 \]
 and
 \[
  \sqrt{\frac{T_2 - S(T_1)}{y - S(x)}} + \sqrt{\frac{y - S(x)}{T_2 - S(T_1)}}\frac{\partial T_2}{\partial y}
   \ge 2\sqrt{\frac{\partial T_2}{\partial y}}.
 \]
 Therefore
 \[
  J_{M \circ (\id,T)}(x,y) \ge \frac1{16}\left(1 + \sqrt{\frac{\partial T_1}{\partial x}}\right)^2\left(1 + \sqrt{\frac{\partial T_2}{\partial y}}\right)^2.
 \]
 Now, in order to obtain \eqref{eq:goal2} it is then sufficient to have
 \[
  \left(1 + \sqrt{\frac{\partial T_1}{\partial x}}\right)\left(1 + \sqrt{\frac{\partial T_2}{\partial y}}\right) \ge \left(1 + \left(\frac{\partial T_1}{\partial x}\frac{\partial T_2}{\partial y}\right)^{\frac14}\right)^2,
 \]
 which immediately follows from
 \[
  \left(\left(\frac{\partial T_1}{\partial x}\right)^\frac14 - \left(\frac{\partial T_2}{\partial y}\right)^\frac14\right)^2 \ge 0.
 \]

 Let us then consider the case $((x,y),T(x,y)) \in D$. By changing to coordinates $\tilde x = \frac1{\sqrt{2}}(x+y)$, $\tilde y = \frac1{\sqrt{2}}(x-y)$
 we may assume that either $\tilde T_1(\tilde x,\tilde y)-\tilde x$ or $\tilde T_2(\tilde x,\tilde y)-\tilde y$ (in the new coordinates) is constant. Assuming the first, we have
 \[
  \frac{\partial \tilde T_1}{\partial \tilde x} = 1 \qquad \text{and}\qquad \frac{\partial \tilde T_1}{\partial \tilde y} = 0
 \]
 giving
 \[
  J_{\tilde T}(\tilde x,\tilde y) = \frac{\partial \tilde T_2}{\partial \tilde y},
 \]
 which is non negative $\mu_0$-almost everywhere in $D$ by cyclical monotonicity \eqref{eq:comcyc}, and
 \[
  J_{\tilde M \circ (\id,\tilde T)}(\tilde x,\tilde y) = \frac12\left(1 + \frac{\partial \tilde T_1}{\partial \tilde x}\right) \cdot\frac12\left(1 + \frac{\partial \tilde T_2}{\partial \tilde y}\right) = \frac12\left(1 + J_{\tilde T}(\tilde x,\tilde y)\right)
 \]
 leading to  \eqref{eq:goal2}.
\end{proof}

Proposition \ref{prop:densityalonggeodesics} then gives the $\CD(0,4)$ condition in $E$. For the convenience of the reader
we now justify here the initial reductions.

\begin{theorem}\label{thm:CD04}
 The space $(E,||\cdot||_\infty,\LL_2\restr{E})$ satisfies $\CD(0,4)$.
\end{theorem}
\begin{proof}
 We have to show that for any $\mu_0,\mu_1 \in \prob{E}$ there exists a geodesic $(\mu_t) \subset \prob E$
 along which we have the estimate
 \begin{equation}\label{eq:entrobound}
 \ent_N(\mu_t)\le (1-t)\ent_N(\mu_0) - t\ent_N(\mu_1) 
 \end{equation}
 for all $N \ge 4$ and $t \in (0,1)$.

 Let us first show that we can obtain this for $t = \frac12$. Take $\epsilon > 0$ and consider the approximated measures
 $\mu_{0,\epsilon} = \rho_{0,\epsilon}\LL_2, \mu_{1,\epsilon} = \rho_{1,\epsilon}\LL_2$
 that are obtained from the measures $\mu_0$ and $\mu_1$ by setting
 \[
  \rho_{i,\epsilon} = \frac{\mu_i\left(E \cap [n\epsilon,(n+1)\epsilon)\times[m\epsilon,(m+1)\epsilon)\right)}{\LL_2\left(E\cap[n\epsilon,(n+1)\epsilon)\times[m\epsilon,(m+1)\epsilon)\right)}
 \]
 on
 \[
  E \cap [n\epsilon,(n+1)\epsilon)\times[m\epsilon,(m+1)\epsilon)
 \]
 for every $n,m \in \Z$, $i=0,1$.

 Now $\ent_N(\mu_{i,\epsilon}) \le \ent_N(\mu_i)$ for all $N > 1$ by Jensen's inequality, $W_2(\mu_{i,\epsilon},\mu_i) \le \epsilon$
 and $\mu_{i,\epsilon} \ll \LL_2$. From $\mu_{0,\epsilon}$ to $\mu_{1,\epsilon}$ there exists an optimal map $T$ given by Proposition \ref{prop:exmap}
 and by Proposition \ref{prop:densityalonggeodesics} we get
 \[
  \ent_N(\mu_{\frac12,\epsilon}) \le \frac12\left(\ent_N(\mu_{0,\epsilon}) + \ent_N(\mu_{1,\epsilon})\right)
                                 \le \frac12\left(\ent_N(\mu_0) + \ent_N(\mu_1)\right),
 \]
 with $\mu_{\frac12,\epsilon} = (M \circ (\id,T))_\sharp\mu_{0,\epsilon}$. Letting $\epsilon \downarrow 0$
 along a subsequence we find a weak limit measure $\mu_\frac12 \in \midp(\mu_0,\mu_1)$ satisfying
 \[
  \ent_N(\mu_\frac12) \le \frac12\left(\ent_N(\mu_0) + \ent_N(\mu_1)\right)
 \]
 for all $N \ge 4$ by the lower semi-continuity of the entropies $\ent_N$.

 Now that we have \eqref{eq:entrobound} at $t = \frac12$ we can continue by taking midpoints between $\mu_0$ and $\mu_\frac12$,
 and between $\mu_\frac12$ and $\mu_1$ and this way obtain \eqref{eq:entrobound} at $t=\frac14$ and $t=\frac34$. Continuing
 iteratively we get \eqref{eq:entrobound} for a dense set of times. Finally, by the lower semi-continuity of $\ent_N$ we have
 \eqref{eq:entrobound} for all $t$, the measures $\mu_t$ being obtained as weak limits of $\mu_s$ as $s \to t$ along the
 dyadic time points.
\end{proof}

\subsection{Failure of the global $\CD(K,\infty)$ condition}

Finally, let us show the calculation implying that the space $X$ does not globally satisfy $\CD(K,N)$.
Because, given any $K \in \R$ and $N \in [1,\infty)$, the $\CD(K,N)$ condition implies the $\CD(K,\infty)$ condition,
it suffices to check the case $N = \infty$.

\begin{theorem}
 Given $K \in \R$, the space $(X,||\cdot||_\infty,\LL_2\restr{X})$ can be constructed
 so that it does not satisfy $\CD(K,\infty)$.
\end{theorem}
\begin{proof} 
 Let $\mu_0 = \frac1{\LL_2(A_0)}\LL_2\restr{A_0}$ and $\mu_1 = \frac1{\LL_2(A_1)}\LL_2\restr{A_1}$ for some sets $A_1, A_2 \subset X$
 with $\LL_2(A_0) = \LL_2(A_1) > 0$ so that every optimal transport between $\mu_0$ and $\mu_1$ transports infinitesimal measures by
 a constant distance $l$. (We can let $A_1$ be a horizontal translation of $A_0$ by $l$.)

 Suppose that the space $(X,||\cdot||_\infty,\LL_2\restr{X})$ satisfies $\CD(K,\infty)$. Then there exists 
 $\mu_\frac12 = \rho_\frac12\LL_2 \in \midp(\mu_0,\mu_1)$ satisfying
 \[
 \ent_\infty(\mu_\frac12)=\int\rho_\frac12\log\rho_\frac12\,\d\LL_2\leq  - \log \LL_2(A_0) - \frac{K}{8}l^2.
 \]
 On the other hand by Jensen's inequality
 \[
 \ent_\infty(\mu_\frac12) \ge - \log \LL_2(A),
 \]
 where $A = \{x \in X \,:\,\rho_\frac12(x) > 0\}$. Therefore
 \[
  \LL_2(A) \ge e^{\frac{K}{8}l^2} \LL_2(A_0),
 \]
 where the multiplicative factor $e^{\frac{K}{8}l^2}$ depends only on $K$ and $l$. Therefore,
 by making the thin part of the space thin enough and taking $A_0$ and $A_1$ to be identical rectangles
 on opposite sides of the thin part, the corresponding midpoint measure does not fit into the thin part
 and we have a contradiction. See Figure \ref{fig:noCD} for an illustration.
\begin{figure}
   \psfrag{a}{$\mu_0$}
   \psfrag{b}{$\mu_1$}
   \psfrag{h}{$h$}
   \psfrag{l}{$l$}
   \psfrag{d}{$< e^{\frac{K}{8}l^2}h$}
   \centering
   \includegraphics[width=0.8\textwidth]{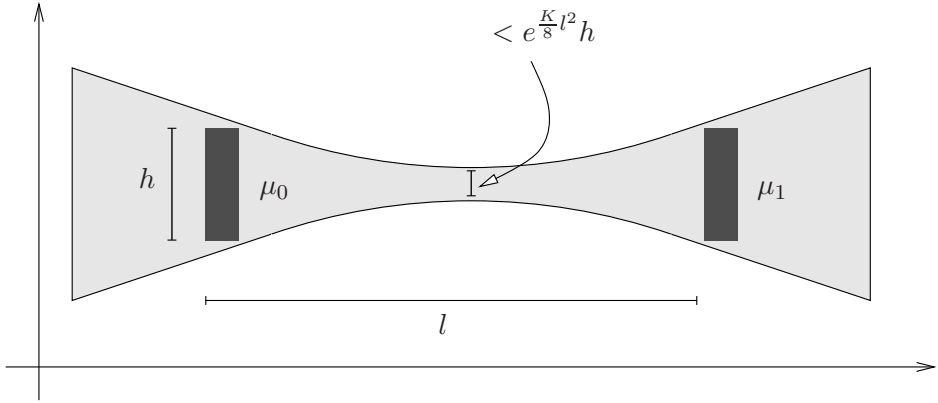}
   \caption{The space fails to satisfy the $\CD(K,\infty)$ condition. The measure $\mu_0$ cannot be transported to $\mu_1$
           without the midpoint measure $\mu_\frac12$ having too small support.}
   \label{fig:noCD}
\end{figure}
\end{proof}

\end{document}